\documentclass[12pt]{article}
\usepackage{amssymb,amsthm,amsmath}
\usepackage{hyperref}
\usepackage{graphicx}
\usepackage[hypcap]{caption}
\usepackage[a4paper, margin=2.4cm]{geometry}

\makeatletter\newtheorem*{rep@theorem}{\rep@title}\newcommand{\newreptheorem}[2]{\newenvironment{rep#1}[1]{\def\rep@title{#2 \ref{##1}}\begin{rep@theorem}}{\end{rep@theorem}}}\makeatother

\newtheorem{theorem}{Theorem}[section]
\newreptheorem{theorem}{Theorem}
\newtheorem{corollary}[theorem]{Corollary}
\newtheorem{lemma}[theorem]{Lemma}
\theoremstyle{definition}
\newtheorem{remark}[theorem]{Remark}

\newcommand{\R}{\mathbb R}

\title{On the number of ordinary conics}
\author{\and Thomas Boys
\and Claudiu Valculescu\thanks{ EPFL, Lausanne, Switzerland. Research partially supported by Swiss National Science Foundation
Grants 200020-144531 and 200021-137574.}\and
Frank de Zeeuw\footnotemark[1]}
\begin{document}
\date{}
\maketitle

\begin{abstract}
We prove a lower bound on the number of ordinary conics determined by a finite point set in $\R^2$.
An \emph{ordinary conic} for $S\subset \R^2$ is a conic that is determined by five points of $S$, and contains no other points of $S$.
Wiseman and Wilson proved the Sylvester-Gallai-type statement that if a finite point set is not contained in a conic, then it determines at least one ordinary conic.
We give a simpler proof of this statement, and then combine it with a theorem of Green and Tao to prove our main result:
If $S$ is not contained in a conic and has at most $c|S|$ points on a line, then $S$ determines $\Omega_c(|S|^4)$ ordinary conics.
We also give constructions, based on the group law on elliptic curves, that show that the exponent in our bound is best possible.
\end{abstract}

\section{Introduction}

The Sylvester-Gallai theorem states that any non-collinear finite set $S\subset \R^2$ determines an \emph{ordinary line}, i.e., a line with exactly two points from $S$.
See \cite{BMP,GT} for the history of this theorem.
It is natural to ask for the minimum number of ordinary lines determined by a non-collinear point set.
Kelly and Moser \cite{KM} proved that a non-collinear $S\subset \R^2$ determines at least $3|S|/7$ ordinary lines,
and Green and Tao \cite{GT} proved that a sufficiently large  $S$ determines at least $|S|/2$ ordinary lines, which is the best possible lower bound because of certain constructions on cubic curves (algebraic curves of degree three).
The proof in \cite{GT} in fact gives a stronger \emph{structural} statement: Any  set $S$ with a linear number of ordinary lines must have most of its points on a cubic curve (see Section \ref{sec:ordinarylines}).

We can ask similar questions for other curves instead of lines.
Elliott \cite{E} proved that every finite $S\subset \R^2$, not contained in a line or circle, determines at least one \emph{ordinary circle}, i.e., a circle containing exactly three points from $S$.
Nassajian Mojarrad and De Zeeuw \cite{NZ} used the structural result from \cite{GT} to prove that such an $S$ determines at least $n^2/4-O(n)$ ordinary circles, which is best possible up to the linear term.

Wiseman and Wilson \cite{WW} proved that any finite $S\subset \R^2$, not contained in a conic, determines at least one \emph{ordinary conic}, i.e., a conic that contains exactly five points of the set,
\emph{and} is determined by these five points.
Here a conic is any zero set of a quadratic polynomial, which may consist of two lines.
Note that for conics, unlike for lines or circles, it makes a difference whether one just requires a conic with five points, or one that is also determined by those five points.
Indeed, if one takes four points on a line and one point off it, then there are infinitely many conics containing these five points (namely the line with the four points combined with any line through the other point),
but they are not determined by the points.
Similarly, five collinear points lie on infinitely many conics.

The proof in \cite{WW}, although elementary, is rather long and intricate.
Devillers and Mukhopadhyay \cite{DM} claimed to give a shorter proof, but we think their proof is incorrect; see Remark \ref{rem:DM} for an explanation.
Recently (and concurrently with our work) Czapli\'{n}ski et al. \cite{Polish} gave a shorter proof, using some more sophisticated algebraic geometry.
We also give a short proof, as preparation for our main theorem.

In this paper we prove the following result, making progress on Problem 7.2.7 of \cite{BMP}.

\begin{theorem}
\label{thm:main}
Let $0<c<1$.
Let $S\subset \R^2$ be a finite set that is not contained in a conic, and that has at most $c|S|$ points on a line.
Then $S$ determines $\Omega_c(|S|^4)$ ordinary conics.
\end{theorem}

It is necessary to put some condition on the maximum number of collinear points in $S$, because of the following construction.
Let $S$ consist of $|S|-k$ points on a line $\ell$ and $k$ points off $\ell$.
Then an ordinary conic contains at most two points from $\ell$,
so the number of ordinary conics is bounded by
$\binom{|S|-k}{2}\binom{k}{3}+(|S|-k)\binom{k}{4}+\binom{k}{5}$.
For $k=o(|S|^{2/3})$, this quantity is $o(|S|^4)$, so the bound in Theorem \ref{thm:main} would not hold.

The exponent $4$ is best possible in this theorem.
This follows from constructions based on group laws on cubic curves, which we will introduce in Section \ref{sec:construction}.

\begin{theorem}\label{thm:construction}
There exist arbitrarily large finite sets $S\subset \R^2$, not contained in a conic and with no four on a line, that determine no more than $\frac{1}{24}|S|^4$ ordinary conics.\\
There exist arbitrarily large finite sets $S\subset \R^2$, not contained in a conic and with at most $|S|/2$ points on a line,
that determine no more than $\frac{7}{384}|S|^4+O(|S|^3)$ ordinary conics.
\end{theorem}

The bound in Theorem \ref{thm:main} depends polynomially on $c$, and it deteriorates as $c$ increases.
To compare it with the upper bound in Theorem \ref{thm:construction}, we prove a simplified version of Theorem \ref{thm:main} in Theorem \ref{thm:generalposition}; it states that if no three points of $S$ are collinear, and $S$ is not contained in a conic, then $S$ determines $\frac{1}{120}|S|^4 - O(|S|^3)$ ordinary conics.
Thus even in this ideal situation we do not quite have a tight bound.
We are not sure where the truth lies, but finding it probably requires a more refined analysis of the extremal configurations of \cite{GT} and the role that they play in our proof;
this is how the tight bound in \cite{NZ} was obtained.

To prove Theorem \ref{thm:main}, we first give a simplified proof of the existence theorem of Wiseman and Wilson, which we then modify to prove our bound on the number of ordinary conics.
We initially follow the setup of \cite{WW}, but then we use very different arguments.
The common element is that we use the Veronese map from $\R^2$ to $\R^5$, defined by
\[V(x,y) = (x,y,x^2,xy,y^2).\]
For three carefully chosen points $p,q,r$, both proofs define a map that sends $V(S)\backslash \{p,q,r\}$ to a non-collinear set $S'$ in some plane in $\R^5$.
By the Sylvester-Gallai theorem, in that plane there is an ordinary line for $S'$ with two points corresponding to $s,t\in S$.
It is then shown that $p,q,r,s,t$ determine an ordinary conic.
In \cite{WW}, the map to the plane consists of three successive central projections from the points $V(p),V(q),V(r)$, and most of their proof is spent analyzing the interactions between these three projections.
We, instead, use a single ``hyperprojection'' from the plane spanned by $V(p),V(q),V(r)$, which allows us to see the properties of the map more clearly.
This leads to a shorter and simpler proof, that, moreover, can be extended to a proof of Theorem \ref{thm:main}.
We also note that Wiseman and Wilson rely on a result of Motzkin \cite{Mo} regarding planes in $\R^3$,
which our proof does not.

Czapli\'nski et al. \cite{Polish} use a different map in a similar way,
namely the \emph{Cremona transformation} based at three non-collinear $p,q,r\in \R^2$.\footnote{If the three points are the fundamental points of the projective plane $\mathbb{RP}^2$, then the Cremona transformation is given by $(x:y:z)\mapsto (yz:xz:xy)$ in projective coordinates.}
This map has very similar properties to the combination of Veronese map and projections that are used in \cite{WW} and in our proof;
we suspect that in some sense these maps are the same.
Another difference is that the proof in \cite{Polish} uses Hirzebruch's inequality (which has only been proved using heavy tools), while \cite{WW} and we use the simple Melchior's inequality.

A common feature of the proofs in \cite{WW,Polish} is that the points $p,q,r$ have to be chosen so that two of their spanned lines are ordinary.
This is an obstacle to obtaining many ordinary conics, because there need not be many such triples in $S$.
In our existence proof, we also use such a triple for convenience, but our proof is flexible enough that we can do without this property.
Thus, in our proof of Theorem \ref{thm:main}, we follow the outline above for every non-collinear triple $p,q,r$.
If there are not too many points on a line, we can obtain $\Omega(|S|^3)$ such triples, and thanks to the Green-Tao theorem\footnote{The bound of Kelly and Moser would suffice quantitatively, but our proof requires the structural statement of Green and Tao.},
we will each time find $\Omega(|S|)$ ordinary conics.
Hence the $\Omega(|S|^4)$.

We cannot help but mention one related conjecture.
In \cite{WW,BMP, Polish}, it was asked whether similar statements hold for curves of higher degree:
Given a degree $d$, if a finite $S\subset \R^2$ does not lie on a curve of degree $d$,
does $S$ determine an \emph{ordinary curve} of degree $d$, i.e., a curve containing exactly $d(d+3)/2$ points of $S$ (and perhaps also determined by those points)?

\section{Ordinary lines}\label{sec:ordinarylines}

In this section we collect several facts about ordinary lines, which we will use in our main proofs.
We include some further discussion, because we think that these lemmas raise some interesting questions about ordinary lines that are not completely solved.

\begin{lemma}\label{lem:twolines}
If $S\subset \R^2$ is finite and not collinear, then there exists a point $p\in S$ that is contained in two lines, each with exactly two or three points of $S$.
\end{lemma}
\begin{proof}
Let $T$ be the total number of lines determined by $S$, and $N_k$ the number of lines containing exactly $k$ points of $S$.
Melchior's inequality \cite{Me} (also used in \cite{GT,KM}) states that $N_2\geq3+\sum_{k\geq 4}(k-3)N_k$, which gives us
\[2N_2+ 2N_3 \geq 2 N_2+N_3\geq 3+N_2+N_3+\sum_{k\geq 4}(k-3)N_k\geq T+3.\]
Since $S$ is not collinear, we have $T\ge |S|$ (by the De Bruijn-Erd\H{o}s theorem \cite{DBE}), so
\[N_2 + N_3\geq \frac{T+3}{2}\geq\frac{|S|+3}{2}.\]
This implies that there is at least one point $p\in S$ lying on two lines, each of which contains two or three points of $S$.
\end{proof}

There are point sets in which no point lies on two ordinary lines, namely the sets $X_{2(2k+1)}$ described by Green and Tao \cite{GT}.
It can be seen from their proof that, for large point sets,
the sets $X_{2(2k+1)}$ are the only sets with this property (up to projective equivalence).
Indeed, 
either $S$ has more than $|S|/2$ ordinary lines and then trivially two pass through the same point;
or else $S$ must be equivalent to one of the sets $X_n$, and by inspection only those of the form $X_{2(2k+1)}$ do not have a point on two ordinary lines.

It is thus also not true that there must always be a triple of points in $S$ such that all three lines spanned by the triple are ordinary (an \emph{ordinary triangle}).
Clearly this fails if $S$ is equivalent to $X_{2(2k+1}$, or if $S$ is contained in two lines, but even if one excludes these specific types of sets, there need not be an ordinary triangle.
Indeed, choose any $k$ points and a disjoint line $\ell$, and for every line spanned by the $k$ points, add its intersection point with $\ell$ to the point set.
Then any triple of points must have a pair from the $k$ points or from $\ell$ that would not span an ordinary line.
It is unfortunate that there need not be an ordinary triangle, because that would have made it much easier to prove the existence of an ordinary conic (see the proof of Theorem \ref{thm:generalposition}).

Erd\H os \cite{Er} wrote that he ``thought for a moment'' that if $S$ has no four on a line, then $S$ must have an ordinary triangle; the constructions above would certainly be excluded.
However, as Erd\H os pointed out, 
F\"uredi and Pal\'asti \cite[Section 8]{FP} gave a construction (based on a construction of Burr, Gr\"{u}nbaum, and Sloane \cite{BGS}) for which this fails; this construction is related to the cubic curve constructions in \cite{GT}.
Another weakened version that one could ask is the following: If the point set is not covered by two lines, must there be a triple such that all three spanned lines have at most three points?
In \cite{GWOP}, the Szemer\'edi-Trotter theorem was used to prove that there is a constant $C$ such that any $S\subset\R^2$, not covered by $C$ lines, has a triple of points with all three lines containing at most $C$ points.

\bigskip

The following lemma tells us when we can find an ordinary line for $S$ that avoids some fixed point $p$, which may or may not be in $S$.
 We say that $S\subset \R^2$ is \emph{near-collinear} if all the points of $S$ but one lie on a line.

\begin{lemma}\label{lem:avoid}
If $S\subset\mathbb{R}^2$ is not near-collinear or collinear, then for every point $p$ of the plane, one can find an ordinary line for  $S$ that avoids $p$.
\end{lemma}
\begin{proof}
Assume that $S$ is not near-collinear or collinear, and that every ordinary line of $S$ passes through $p$.
We note that the following proof works whether $p$ is in $S$ or not.

By the theorem of Kelly and Moser \cite{KM} mentioned in the introduction,
the set $S\cup  p$ determines at least $3|S|/7$ ordinary lines.
Each of these lines must pass through $p$, and each contains one other point.
 Let $S_1$ be the set of these other points, so $|S_1|\ge 3|S|/7$.
 Since $S$ is not near-collinear or collinear, $S\backslash p$ is not collinear, so by another application of the Kelly-Moser theorem,
 the set $S\backslash p$ determines at least $3(|S|-1)/7$ ordinary lines.
 Again each line must pass through $p$, but now each contains two points of $S\backslash p$.
Let $S_2$ be the set of these other points, so $|S_2|\ge 6(|S|-1)/7$.
The sets $S_1$ and $S_2$ are disjoint, and together they contain $(9|S|-6)/7$ points,
which is more than $|S|$ when $|S|>3$.
\end{proof}

We think it would be interesting to see a direct proof of Lemma \ref{lem:avoid}, not using the result of Kelly and Moser.
One could also ask for the minimum number of ordinary lines that avoid $p$,
or ask for an ordinary line that avoids several points.
These questions can be answered using the results of Green and Tao, but this only holds for sufficiently large point sets.
In Lemma \ref{lem:avoidset} we prove that, for sufficiently large sets, there are many ordinary lines avoiding any number of fixed points.
This lemma will be crucial in the proof of our bound on the number of ordinary conics.

We first  state the result of Green and Tao \cite{GT}.
The second part of $(b)$ is not explicitly stated in \cite{GT}, but in \cite{NZ} it is extracted from the proof in \cite{GT}.

\begin{theorem}[Green-Tao]\label{thm:greentao}
For every $K\in\mathbb{R}$ there exists an $N_K \in \mathbb{N}$ such that the
following statements hold for any non-collinear set $S$ of at least $N_K$ points in $\mathbb{R}^2$.
\begin{enumerate}
\item[$(a)$] The set $S$ determines at least $|S|/2$ ordinary lines.
\item[$(b)$] If $S$ determines fewer than $K|S|$ ordinary lines,
then either all but $O(K)$ points of $S$ lie on a line,
or all but $O(K)$ points of $S$ lie on a cubic curve.
In the second case, at least $|S|/2-O(K)$ of the ordinary lines for $S$ are tangent lines to the cubic.
\end{enumerate}
\end{theorem}

We need the following simple lemma, stated in \cite{NZ}.

\begin{lemma}\label{lem:tangents}
Let $C$ be an algebraic curve of degree $d$ in $\mathbb{R}^2$ and $p\in\mathbb{R}^2$ a fixed point.
Then at most $d(d-1)$ lines through $p$ are tangent to $C$.
\end{lemma}

We are now ready to prove that for sufficiently large sets $S$ there are ordinary lines for $S$ avoiding any point set $T$.
Note that the points of $T$ may or may not be in $S$.
It is necessary to exclude the case that $S\backslash T$ is collinear, since in that case there would be no ordinary lines avoiding $T$.

\begin{lemma}\label{lem:avoidset}
For every $k$ there exist $L_k, M_k\in \mathbb{N}$ such that the following holds.
Given a set $S\subset \R^2$ of at least $L_k$ points, and a set $T\subset\R^2$ of $k$ points such that $S\backslash T$ is not collinear,
there are at least $|S|/2-M_k$ ordinary lines for $S$ that avoid $T$.
\end{lemma}
\begin{proof}
Set $L_k = N_{k+1}$ (a constant from Theorem \ref{thm:greentao}) and assume that $|S|\geq L_k$.
Suppose that $S$ determines at least $(k+1)|S|$ ordinary lines.
Since each of the $k$ points of $T$ is hit by at most $|S|$ of the lines determined by $S$, at most $k|S|$ of the ordinary lines hit $T$.
Hence there are at least $|S|$ ordinary lines that avoid $T$.

Suppose that $S$ determines fewer than $(k+1)|S|$ ordinary lines.
Then, by Theorem \ref{thm:greentao}$(b)$, all but $O(k)$ points of $S$ lie on a line $L$ or on a cubic curve $C$.
In the first case, 
there is a point $p\in S$ that is not on $L$ and not in $T$.
Then there are at least $|S| - O(k)$ ordinary lines through $p$ that avoid $T$.
In the second case,
there are at least $|S|/2-O(k)$ ordinary lines for $S$ that are tangent lines to $C$.
If $q$ is a point in $T$,
then by Lemma \ref{lem:tangents}, at most six tangent lines of $C$ hit $q$.
Altogether, at most $6k$ ordinary tangent lines for $S$ hit $T$, so there are at least $|S|/2-O(k)$ ordinary lines for $S$ that avoid $T$.
\end{proof}

\section{Conics and the Veronese map}

Let us first recall that a \emph{conic} is the zero set in $\mathbb{R}^2$ of a polynomial of degree $2$.
Note that under this definition, a conic can be a union of two distinct lines (for instance $xy=0$), and this type of conic will play a role in our proof.
A conic can also be an empty set ($x^2+y^2 = -1$) or a double line ($x^2 = 0$), but these two types will not play a serious role in our proof.
A set $S\subset\mathbb{R}^2$ is \emph{co-conic} if it is contained in a conic.
Given a finite set $S\subset \R^2$, a conic is ordinary for $S$ if it contains exactly five points of $S$, and is the only conic containing these five points.
A double line (let alone an empty set) cannot be ordinary, since if it contains five points of $S$, there are infinitely many other conics containing these five points.

 We will use the \emph{Veronese map} $V:\mathbb{R}^2\to\mathbb{R}^5$, defined by
 \[V:(x,y)\mapsto (x,y,x^2,xy,y^2).\]
This map defines a bijection between conics in $\mathbb{R}^2$ and hyperplanes in $\mathbb{R}^5$ in the following sense:
A set $S\subset \R^2$ lies on the conic $ a_0 + a_1 x + a_2 y + a_3 x^2 + a_4 xy + a_5 y^2=0$
if and only if $V(S)\subset \R^5$ lies on the hyperplane $a_0+a_1z_1 +a_2z_2 +a_3z_3+a_4z_4+a_5z_5=0$.
In what follows, we prove a series of properties of this mapping; all of these were also proved in some form in \cite{WW}.

We use the term \emph{$k$-flat} to refer to an affine subspace of dimension $k$.
For $S\subset \R^d$, we write $\overline{S}$ for the smallest affine subspace containing $S$,
and given points $p_1,\ldots, p_k$, we write $\overline{p_1\cdots p_k}$ for the smallest affine subspace containing them.
For instance, $\overline{pq}$ is the line spanned by $p$ and $q$.

A $k$-flat is \emph{ordinary} for a set $S$ if it contains exactly $k+1$ points of $S$ and is the unique $k$-flat through these points.\footnote{This is not the standard definition of ordinary $k$-flat; see for instance Motzkin \cite{Mo}.}
With this definition, we have the following key fact:
\begin{center}
\emph{A conic is ordinary for $S\subset\R^2$ if and only if the corresponding hyperplane is ordinary for $V(S)\subset \R^5$.}
\end{center}
In the proofs of our main theorems, our goal will be to find ordinary hyperplanes for $V(S)$.

Note that a double line cannot be an ordinary conic,
and the corresponding hyperplane is not ordinary, because the intersection of $V(\R^2)$ with this hyperplane is contained in other hyperplanes; see Lemma \ref{lem:twoflats} below.
For instance, the double line defined by $x^2 = 0$ corresponds to the hyperplane $z_3=0$, and the points of $V(\R^2)$ that satisfy $z_3=0$ also satisfy $z_1=0$.

\begin{lemma}\label{lem:oneflats}
No three points of $V(\R^2)$ lie on the same line.
\end{lemma}
\begin{proof}
A point $(z_1,z_2,z_3,z_4,z_5)\in \mathbb{R}^5$ belongs to $V(\mathbb{R}^2)$ if and only if it satisfies
\[z_3=z_1^2,~~~ z_4=z_1z_2,~~~ z_5=z_2^2.\]
For each of these equations, a line in $\R^5$ either has at most two solutions, or all its points satisfy the equation.
Thus we are done unless the line is contained in $V(\R^2)$.

We show that $V(\R^2)$ contains no line, which finishes the proof.
Consider a parametrization $p(t)=(p_1(t),\dots ,p_5(t))$ of the line, with all $p_i$ linear, and suppose it satisfies all three equations for all $t$.
We cannot have both $p_1$ and $ p_2$ constant, since then the other three coordinates would also be constant.
But if, say, $p_1$ is not constant, then $p_3(t)=p_1(t)^2$, which is not possible.
\end{proof}

\begin{lemma}\label{lem:twoflats}
Let $S\subset\mathbb{R}^2$ with $|S|\ge 4$.
Then $V(S)$ is contained in a $2$-flat if and only if $S$ is contained in a line.
\end{lemma}
\begin{proof}
Suppose that $S$ is contained in the line defined by $y =ax+b$ (vertical lines can be handled similarly).
Then the points of $V(S)$ have the form $(x,ax+b, x^2, ax^2+bx, a^2x^2+2abx+b^2)$.
Thus they satisfy the three linear equations $z_2=az_1+b$, $z_4=az_3+bz_1$, and $z_5=a^2z_3+2abz_1+b^2$.
These three linearly independent equations define a $2$-flat that contains $V(S)$.

Suppose that $S$ is not collinear.
Then there is a subset $S'$ of four points of $S$ that determine two distinct lines $\ell_1$ and $\ell_2$.
We can pick a point $p$ on one of the lines such that the set $S'\cup p$ consists of five points that uniquely determine the conic $\ell_1\cup\ell_2$.
This implies that $V(S'\cup p)$ determines a unique hyperplane, and thus does not lie in a $3$-flat.
Then $V(S')$ does not lie in a $2$-flat.
\end{proof}

\begin{corollary}\label{cor:twoflats}
If the points $p,q,r\in S\subset {\mathbb{R}^2}$ are not collinear, then $\overline{V(p)V(q)V(r)} $ is an ordinary $2$-flat for $V(S)$.
\end{corollary}
\begin{proof}
By Lemma \ref{lem:oneflats}, the points $V(p), V(q), V(r)$ are not collinear, so $\overline{V(p)V(q)V(r)} $ is indeed a $2$-flat.
Suppose that $\overline{V(p)V(q)V(r)} $ is not ordinary, so it contains another point $V(s)$ from $V(S)$.
Then, by applying Lemma \ref{lem:twoflats} to the four points $p,q,r,s$, the points $p,q,r\in S$ are collinear, contrary to the assumption.
\end{proof}

The next lemma states that the near-collinear sets are the only large sets that determine more than one conic.

\begin{lemma}\label{lem:threeflats}
Let $S\subset\mathbb{R}^2$ be a non-collinear set with  $|S|\ge 5$.
Then $V(S)$ is contained in a $3$-flat if and only if $S$ is \emph{near-collinear}.
\end{lemma}
\begin{proof}
Suppose $V(S)$ is contained in a $3$-flat.
Then $V(S)$ lies in the intersection of two hyperplanes, so $S$ lies in the intersection of the two corresponding conics, which we label $C_1,C_2$.
By B\'{e}zout's inequality, since $|S|\ge 5$, $C_1$ and $C_2$ are both reducible and share a line $\ell$.
We write $C_i=\ell\cup\ell_i$, with $\ell_1$, $\ell_2$ distinct lines.
Thus $C_1\cap C_2=\ell\cup(\ell_1\cap\ell_2)$,
which is a line and a point (or a line, but $S$ is assumed non-collinear),
so $S$ is \emph{near-collinear}, which completes the proof.

Suppose $S$ is near-collinear, so there is a $p\in S$ such that $S\backslash  p$ is collinear.
Then, by Lemma \ref{lem:twoflats}, $V(S\backslash p)$ is contained in a $2$-flat, so $V(S)$ is contained in a $3$-flat.
\end{proof}

\section{The existence of an ordinary conic}

In this section we prove the existence of an ordinary conic using the lemmas proved in the previous two sections.
Note that we do not use the Green-Tao theorem in this section.

As mentioned in the introduction, we make use of ``hyperprojections''.
These are generalizations of the more familiar central projections,
which take a fixed point $p\in\R^D$ and a $(D-1)$-flat $Q$ not containing $p$, and map any point $x\neq p$ to the point $\overline{px}\cap Q$; to make this well-defined, one needs to work in projective space $\mathbb{PR}^D$, but even in $\R^D$ the map is well-defined for most points $x$.
This can be generalized to a hyperprojection from a $k$-flat $P$ in $\R^D$ to a $(D-k-1)$-flat $Q$.
Indeed, for most $x\notin P$, the set $\overline{P\cup x}\cap Q$ should consist of a single point.
See for instance Harris \cite[page 37]{H}.

For us, the goal is to map a given finite set $S$ in $\R^5$ from a $2$-flat $P$ to a $2$-flat $Q$.
Because we are only mapping a finite set $S$,
this map should be well-defined for a ``generic'' choice of $Q$.
In the following lemma we make this precise.

\begin{lemma} \label{lem:generic}
 Given a $2$-flat $P\subset\mathbb{R}^5$ and a finite set $S\subset\mathbb{R}^5$ that is disjoint from $P$, there is a $2$-flat $Q\subset\mathbb{R}^5$ disjoint from $P$ such that, for every $x\in S$, $\overline{P\cup x}\cap Q$ consists of a single point.
\end{lemma}
\begin{proof}
We work in projective space for the duration of this proof.
We use the word \emph{flat} and the bar notation also for projective subspaces.
We think of $\mathbb{PR}^5$ simply as a copy of $\R^5$ together with a projective $4$-flat $H$ at infinity.
The crucial fact is that any $k$-flat and $(D-k)$-flat in $\mathbb{PR}^D$ must intersect.
The given flat $P$ and set $S$ can be naturally embedded in $\mathbb{PR}^5$.
We will obtain a $2$-flat $Q\subset \mathbb{PR}^5$ that has the required property, and we will observe that this gives a $2$-flat in $\R^5$ with the same property.

For each $x\in S$, $\overline{P\cup x}$ is a $3$-flat, containing a $2$-flat at infinity, which we label $F_x$.
We construct a line at infinity that avoids $F_x$ for all $x\in S$.
In order to do this, we first pick an arbitrary point $p\in H\backslash \bigcup_{x\in S}F_x$.
Then $\overline{F_x\cup p}$ is a $3$-flat for all $x\in S$,
so we can still pick a point $q\in H$ that is not in $\bigcup_{x\in S} \overline{F_x\cup p}$.
We claim that the line $\overline{pq}$ does not intersect any of the $2$-flats $F_x$.
Indeed, if $\overline{pq}\cap F_x\neq\emptyset$, then, since also $p\in \overline{F_x\cup p}$,
the whole line $\overline{pq}$ is contained in $\overline{ F_x \cup p}$, which would contradict the choice of $q$.

Choose $r\in \R^5$ such that $r$ does not lie in the projective $4$-flat $\overline{P\cup \{p,q\}}$.
Set $Q=\overline{pqr}$,
so that $Q$ is a $2$-flat disjoint from $P$ in $\mathbb{PR}^5$.
Indeed, if $x\in P\cap Q$, then $r\in \overline{P\cup \{p,q\}} $, contradicting the choice of $r$.
Moreover, for any $x\in S$, $Q$ and $\overline{P\cup x}$ do not intersect at infinity, because $Q\cap H = \overline{pq}$, and $\overline{pq}$ was constructed to be disjoint from $\overline{P\cup x}$.
Note that since $Q$ contains $r$, its intersection with $\R^5$ is also a $2$-flat.

We prove that $\overline{P\cup x}\cap Q$ consists of a single point for every $x\in S$, which finishes the proof.
The $2$-flat $Q$ and the $3$-flat $\overline{P\cup x}$ must intersect in $\mathbb{PR}^5$.
Since $Q$ and $\overline{P\cup x}$ do not intersect at infinity,
they must intersect in $\mathbb{R}^5$.
Suppose that $Q$ and $\overline{P\cup x}$ intersect in more than one point, so their intersection contains a line $\ell$.
Then the $1$-flat $\ell$ and the $2$-flat $P$ are contained in the $3$-flat $\overline{P\cup x}$, so they must intersect in $\mathbb{PR}^5$.
On the other hand, $\ell$ is also contained in $Q$, so $P$ and $Q$ intersect, which is a contradiction.

\end{proof}

\begin{remark}\label{rem:hyperprojection}
Lemma \ref{lem:generic} tells us that, given a $2$-flat $P$ and a finite set $S$ disjoint from $P$, there is a $2$-flat $Q_{P,S}$ and a well-defined hyperprojection $\pi_{P,S}:S\rightarrow Q_{P,S}$ such that $\pi_{P,S}(x)$ is the point in $\overline{P\cup x}\cap Q_{P,S}$.
We will suppress the dependence on $P$ and $S$ when it is clear from the context.
\end{remark}

In the following lemma we investigate the injectivity of this hyperprojection on the Veronese image of a finite set of points in the plane.
We write $\Delta_{pqr}=\overline{pq} \cup \overline{pr} \cup \overline{qr}$ for the extended triangle spanned by $p,q,r$.
When we say that a map from $A$ to $B$ is injective on $A'\subset A$, we mean that no two elements of $A'$ are mapped to the same element of $B$.

\begin{lemma}\label{lem:injective}
Let $p,q,r\in \mathbb{R}^2$ be non-collinear and $P=\overline{V(p)V(q)V(r)}$.
Then for any finite set of points $S\subseteq\mathbb{R}^2\backslash\{p,q,r\}$, the hyperprojection\footnote{We abuse notation slightly, because $\pi$ is not defined at some points, including $V(p),V(q),V(r)$.
When applying a function to a set, we will simply ignore elements of the set at which the function is not defined.} $\pi_{P,V(S)}:V(S)\to Q_{P,V(S)}$ from  Remark \ref{rem:hyperprojection}
is injective on $V\left(S\backslash \Delta_{pqr}\right)$ (see Figure \ref{fig:fig1}).

Moreover, there are $\alpha_1,\alpha_2,\alpha_3\in Q_{P,V(S)}$ such that $\pi(V(\Delta_{pqr}))= \{\alpha_1,\alpha_2,\alpha_3\}$, and more precisely we have $\pi\left(V(\overline{pq})\right)=\{\alpha_1\}$,
$\pi\left(V(\overline{pr})\right)=\{\alpha_2\}$,
and $\pi\left(V(\overline{qr})\right)=\{\alpha_3\}$.
\end{lemma}

\begin{figure}[ht]
\centering
\includegraphics[height=2.5in]{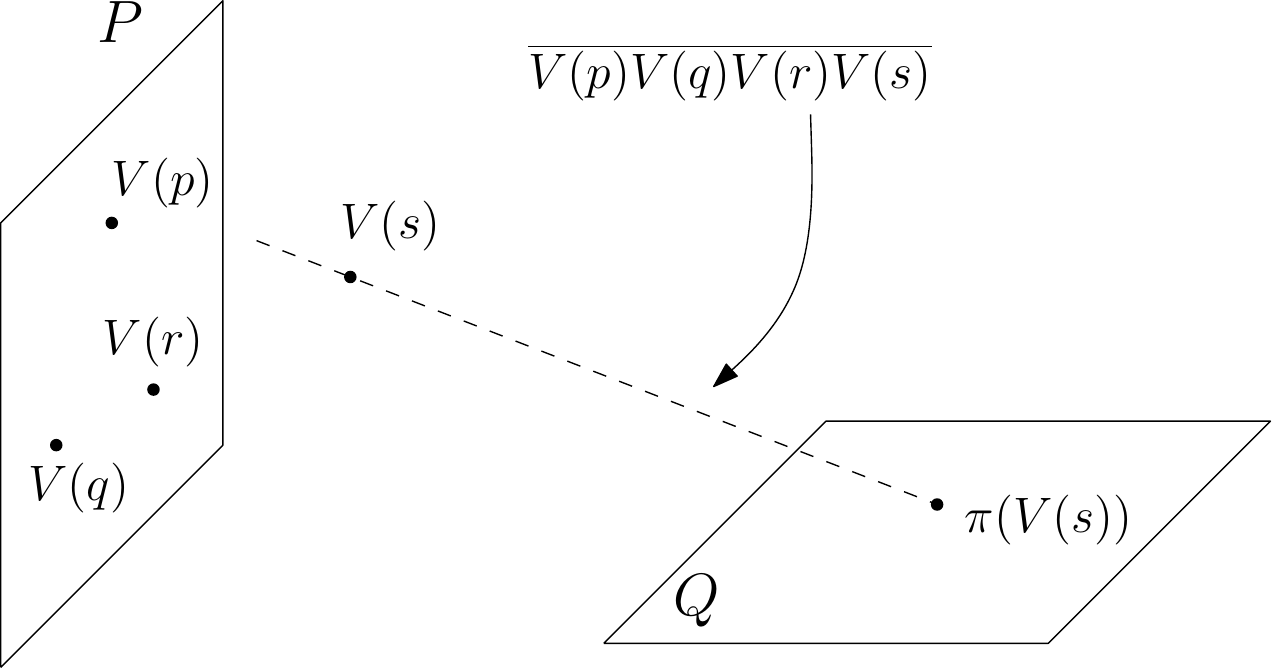}
 \caption{The projection $\pi$. The dashed line represents the $3$-flat spanned by $P$ and $V(s)$.}
 \label{fig:fig1}
\end{figure}

\begin{proof}
Note that by Lemma \ref{lem:twoflats} and the fact that $p,q,r$ are non-collinear, $P$ does not contain any other points of $V(S)$, so the hyperprojection is well-defined on $V(S\backslash\{p,q,r\})$

Suppose there are distinct $s,t\in S\backslash \Delta_{pqr}$ such that $\pi(V(s))=\pi(V(t))=\alpha$,
so the points $V(p), V(q),V(r),V(s),V(t)$ all lie on the $3$-flat $\overline{P\cup \alpha}$.
From Lemma \ref{lem:threeflats} we obtain that $p,q,r,s,t$ form a near-collinear configuration (see Figure \ref{fig:fig2}).
It follows that $s$ and $t$ lie in $\Delta_{pqr}$, which is a contradiction.
Therefore, $\pi$ is injective on $V\left(S\backslash \Delta_{pqr}\right)$.

To prove the second part of the lemma,
it suffices to show that for any $s,t\in \overline{pq}\backslash\{p,q\}$ we have $\pi(V(s)) = \pi(V(t))$.
The set $T=\{p,q,r,s,t\}$ is near-collinear,
so by Lemma \ref{lem:threeflats},
$V(T)$ is contained in a $3$-flat,
whose intersection with $Q$ is a single point $\alpha\in Q$.
Then we have $\pi(V(s))=\alpha=\pi(V(t))$, which completes the proof.
\end{proof}

\begin{figure}[ht]
\centering
\includegraphics[height=2.5in]{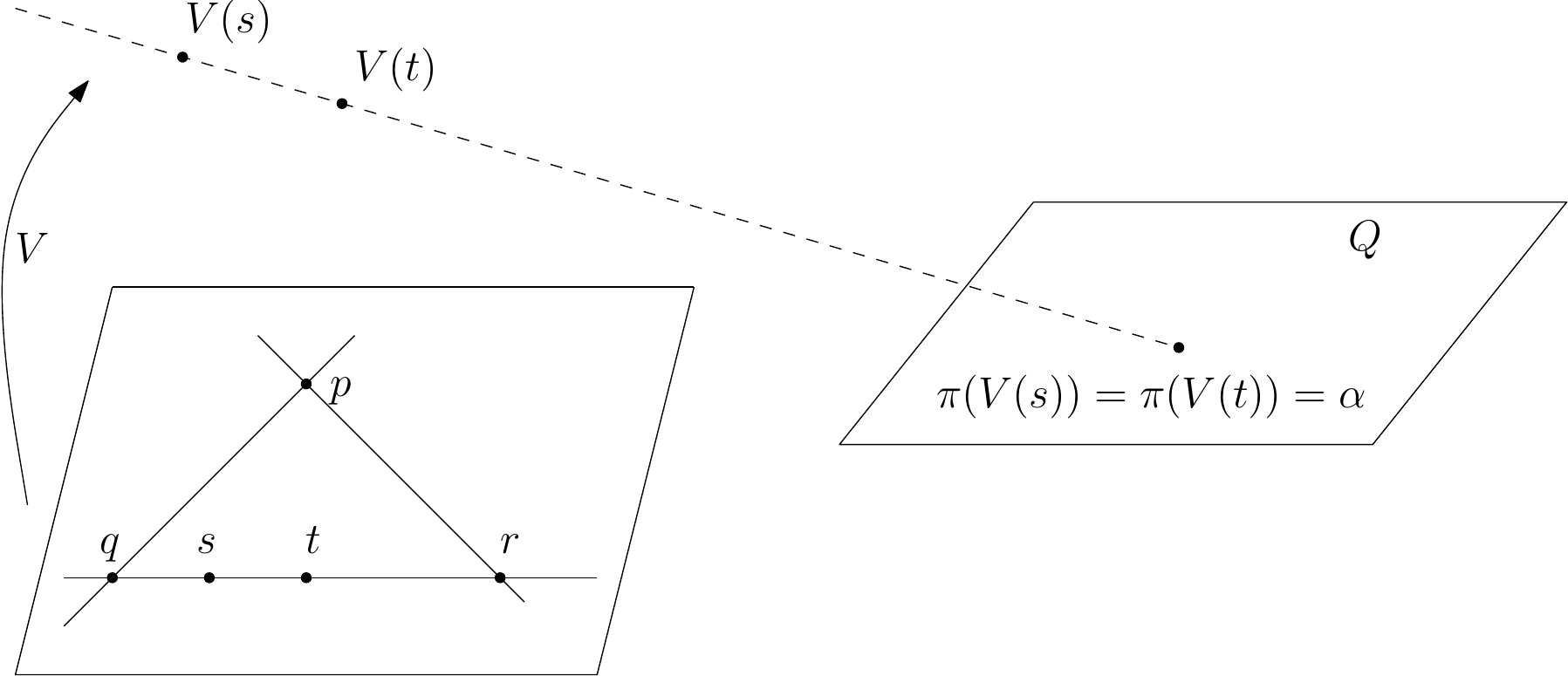}
 \caption{Non-injectivity of $\pi$ on $\Delta_{pqr}$.}\label{fig:fig2}
\end{figure}

We are now fully equipped to prove the existence theorem.

\begin{theorem}\label{thm:existence}
A finite set $S\subset\R^2$ that is not contained in a conic determines at least one ordinary conic.
\end{theorem}
\begin{proof}
If $S$ determines a line $\ell$ with exactly three points from $S$, then we are done.
Indeed, since the set $S$ is not contained in a conic, the set $S\backslash \ell$ is not collinear.
Hence, by the Sylvester-Gallai theorem, there is an ordinary line $\ell'$ for $S\backslash \ell$, so $\ell\cup \ell'$ is an ordinary conic for $S$.
Hence we can assume that no line contains exactly three points of $S$.

We wish to find an ordinary hyperplane for $V(S)\subset\mathbb{R}^5$.
Since no line contains exactly three points of $S$, Lemma \ref{lem:twolines} gives us three points $p,q,r\in S$ so that the lines $\overline{pq}$ and $\overline{pr}$ are ordinary for $S$.
Set $P = \overline{V(p)V(q)V(r)}$; by Corollary \ref{cor:twoflats} and the fact that $p,q,r$ are non-collinear, $P$ contains no other points of $V(S)$.
Lemma \ref{lem:generic} then gives us a $2$-flat $Q\subset\mathbb{R}^5$ disjoint from $P$ and a well-defined hyperprojection $\pi:V(S)\to Q$. Moreover, since the lines $\overline{pq}$ and $\overline{pr}$ are ordinary,
Lemma \ref{lem:injective} tells us that $\pi$ is injective on $V\left(S\backslash\overline{qr}\right)$,
and that there is at most one point $\alpha=\pi(V(\overline{qr}))\in Q$ for which $\pi^{-1}(\alpha)\cap V(S)$ consists of more than one point.
Note that if we also had $\overline{qr}\cap S = \{q,r\}$,
then $\pi$ would be injective on $V(S)$.
However, as mentioned in Section \ref{sec:ordinarylines}, we cannot guarantee that $S$ determines an ordinary triangle, i.e., a triple $p,q,r$ with all three lines ordinary.

Let $B=\pi(V(S))$.
Note that $B$ cannot be contained in a line, since otherwise $V(S)$ would lie on a hyperplane, which is equivalent to all the points of $S$ being contained in a conic.
Hence, there is at least one ordinary line for $B$.
If there is an ordinary line that avoids $\alpha$, say $\overline{\sigma\tau}$ with $\sigma, \tau\in B$,
then, since $\pi$ is injective away from $\alpha$, there are unique points $s,t\in S$ such that $\pi(V(s)) = \sigma$ and $\pi(V(t)) = \tau$.
In this case $\overline{V(p)V(q)V(r)V(s)V(t)}$ is an ordinary hyperplane for $V(S)$ and we are done.

We are left with the case where no ordinary line for $B$ avoids $\alpha$.
By Lemma \ref{lem:avoid} and the fact that $B$ is not collinear,
$B$ is near-collinear.
In other words, there is a line in $Q$ that contains $B\backslash \alpha = \pi(V(S\backslash \overline{qr}))$.
It follows that $S\backslash\overline{qr}$ is contained in a conic $C$,
so $S$ is contained in the union of $C$ and $\overline{qr}$.
Note that $q,r\in C$.
Since $S$ is not contained in a conic, 
it must contain a point $u\in \overline{qr}\backslash C$ and two points $v,w\in C\backslash\{q,r\}$.
B\'{e}zout's inequality tells us that each of the three lines $\overline{uv},\overline{uw},\overline{vw}$ contains at most three points.
But by the assumption made at the start of the proof, no line has exactly three points of $S$.
Therefore, the triangle $\Delta_{uvw}$ is ordinary, i.e., each of its three lines is ordinary.
In that case, repeating the argument above with $u,v,w$ instead of $p,q,r$, we get a hyperprojection $\pi$ that is injective on $V(S)$.
Then any ordinary line in $Q$ gives an ordinary conic for $S$.
\end{proof}

It might be considered improper to take a \emph{reducible} ordinary conic (consisting of two lines) as we do in the first paragraph of the proof of Theorem \ref{thm:existence}.
However, Czapli\'nski et al. \cite{Polish} give a construction that contains no \emph{irreducible} ordinary conic.

\begin{remark}\label{rem:DM}
As remarked in the introduction,
Devillers and Mukhopadhyay \cite{DM} gave a proof of Theorem \ref{thm:existence} that we think is incorrect.
Their initial setup is similar to that in our proof of Theorem \ref{thm:existence}, but they seem to ignore the fact that the hyperprojection need not be injective.

In fact, their Theorem 1 states that if $S$ is not co-conic, then for every non-collinear $p,q,r$ there would be $3|S|/7$ ordinary conics containing $p,q,r$.
This is false.
Take for instance $|S|-3$ points on a line $\ell$ and three non-collinear points $p,q,r$ off $\ell$.
Then an ordinary conic for $S$ cannot contain $\ell$, so $p$ together with two points $s,t$ from $\ell$ form a non-collinear triple that is contained in exactly one ordinary conic, namely the one spanned by $p,q,r,s,t$.
\end{remark}

\section{The number of ordinary conics}

In this section we prove Theorem \ref{thm:main}.
First we prove a weaker version, where the point set is assumed to be in general position.
We do this as a warmup for our main proof,
and also because this best-case scenario tells us what the best coefficient of $|S|^4$ is that we could expect from our proof.

\begin{theorem}\label{thm:generalposition}
Let $S\subset \R^2$ be a finite set that is not contained in a conic and has no three points on a line.
Then $S$ determines $\frac{1}{120}|S|^4 - O(|S|^3)$ ordinary conics.
\end{theorem}
\begin{proof}
Let $p,q,r$ be any triple from $S$ and set $P = \overline{V(p)V(q)V(r)}$.
The triple is non-collinear,
so by Corollary \ref{cor:twoflats} $P$ contains no other points.
By Remark \ref{rem:hyperprojection}, there is a plane $Q\subset \R^5$ and a well-defined projection $\pi:V(S)\to Q$.
Because $\Delta_{pqr}$ contains no points of $S$ other than $p,q,r$,
Lemma \ref{lem:injective} tells us that $\pi$ is injective on $V(S)$.

By Theorem \ref{thm:greentao}, $\pi(V(S))$ determines at least $(|S|-3)/2 \geq |S|/2-2$ ordinary lines (we can assume $S$ to be sufficiently large by adjusting the constants in the theorem).
If $\overline{\sigma\tau}$ with $\sigma,\tau\in \pi(V(S))$ is such an ordinary line,
then by injectivity of $\pi$ there are unique $s,t\in S$ such that $\pi(V(s)) = \sigma$ and $\pi(V(t)) = \tau$.
Hence $\overline{V(p)V(q)V(r)V(s)V(t)}$ is an ordinary hyperplane for $V(S)$, and $p,q,r,s,t$ determine an ordinary conic for $S$.

We get at least $|S|/2-2$ ordinary conics for each of the $\binom{|S|}{3}$ triples from $S$.
Each ordinary conic is counted $\binom{5}{3}$ times.
Therefore, we find at least
\[\frac{1}{\binom{5}{3}} \left(\frac{1}{2}|S|-2\right) \binom{|S|}{3}  = \frac{1}{120} |S|^4 - O(|S|^3)\]
ordinary conics for $S$.
\end{proof}

The following lemma gives a simple bound on the number of non-collinear triples determined by a point set.
Such a bound must have some condition on the maximum number of points on a line, since, for instance, a near-collinear set $S$ determines only $\binom{|S|-1}{2}$ non-collinear triples.

\begin{lemma}\label{lem:noncollinear}
Let $S\subset \R^2$ be a finite set with at most $c|S|$ points on a line.
Then $S$ determines at least $\frac{1}{3}(1-c)|S|\binom{|S|}{2}$ non-collinear triples.
\end{lemma}
\begin{proof}
Note that we are considering \emph{unordered} triples.
Let $L$ be the set of lines determined by $S$.
We denote the number of points of $S$ on $\ell$ by $p_\ell$.
Then we have $\sum_{\ell\in L}\binom{p_\ell}{2}=\binom{|S|}{2}$.
The number of non-collinear triples equals $\frac{1}{3}\sum_{\ell\in L}(|S|-p_\ell)\binom{p_\ell}{2}$, since given a line with $p_\ell$ points, we get a triangle for each pair of points from the line together with each point off the line.
By the assumption in the hypothesis, there are at most $c|S|$ points on a line,
so $|S|-p_\ell\geq(1-c)|S|$ for each $\ell\in L$.
Thus we have
\[\frac{1}{3}\sum_{\ell\in L}(|S|-p_\ell)\binom{p_\ell}{2}
\geq\frac{1}{3}(1-c)|S| \sum_{\ell\in L}\binom{p_\ell}{2}
=\frac{1}{3}(1-c)|S|\binom{|S|}{2}, \]
which proves the lemma.
\end{proof}

A weaker version of the lemma could be deduced from a theorem of Beck \cite{B}, 
which states that any point set has $\Omega(n)$ points on a line or determines $\Omega(n^2)$ lines.
In both cases it follows that there are $\Omega(n^3)$ non-collinear triples.
 However, the simple and self-contained proof above  provides a better dependence on the constant than \cite{B}.

We restate our main  theorem for the convenience of the reader.

\begin{reptheorem}{thm:main}
Let $0<c<1$.
Let $S\subset \R^2$ be a finite set that is not contained in a conic, and that has at most $c|S|$ points on a line.
Then $S$ determines $\Omega_c(|S|^4)$ ordinary conics.
\end{reptheorem}
\begin{proof}
Let $p,q,r$ be a non-collinear triple from $S$ and set $P = \overline{V(p)V(q)V(r)}\subset \R^5$.
By Corollary \ref{cor:twoflats} and
Remark \ref{rem:hyperprojection}, there is a plane $Q\subset \R^5$ and a well-defined projection $\pi:V(S)\to  Q$,
which by Lemma \ref{lem:injective} is injective on $V(S\backslash \Delta_{pqr})$.
Also, by Lemma \ref{lem:injective}, there are $\alpha_1,\alpha_2,\alpha_3\in Q$ such that $\pi(V\left(\overline{pq}\right))=\alpha_1$, $\pi(V\left(\overline{pr}\right))=\alpha_2$, and $\pi(V\left(\overline{qr}\right))=\alpha_3$.
Set
\[A_{pqr} = S\cap \Delta_{pqr}~~~\text{and}~~~B_{pqr}=\pi(V(S\backslash\Delta_{pqr}))=\pi(V(S))\backslash\{\alpha_1,\alpha_2,\alpha_3\},\]
so that we have $|A_{pqr}|+|B_{pqr}|=|S|$, using the fact that the composition of $V$ and $\pi$ is injective on $S\backslash \Delta_{pqr}$.

Suppose that a non-collinear triple $p,q,r$ has the property that $B_{pqr}$ is not collinear.
By Lemma \ref{lem:avoidset}, there are constants $L,M$ such that the following holds:
If $|B_{pqr}|\geq L$, then there are at least $\frac{1}{2}|B_{pqr}|-M$ ordinary lines for $B_{pqr}$ that avoid $\alpha_1,\alpha_2,\alpha_3$.
Let $\overline{\sigma\tau}$ be such an ordinary line, with $\sigma,\tau\in B_{pqr}$.
By the injectivity of $\pi$ on $V(S\backslash\Delta_{pqr})$,
there exist unique points $s,t\in S$ such that $\pi(V(s)) = \sigma$ and $\pi(V(t)) = \tau$.
Then $\overline{V(p)V(q)V(r)V(s)V(t)}$ is an ordinary hyperplane for $V(S)$, and $p,q,r,s,t$ determine an ordinary conic for $S$.
Therefore, given a non-collinear triple $p,q,r$ with $B_{pqr}$ not collinear and $|B_{pqr}|\geq L$, we get at least $\frac{1}{2}|B_{pqr}|-M$ ordinary conics for $S$ that contain $p,q,r$.

To be able to apply the argument above, we have to show that enough non-collinear triples $p,q,r$ have  $B_{pqr}$ non-collinear and $|B_{pqr}|$ large.
 We distinguish two cases: either $|B_{pqr}|\ge (1-c)^2|S|$ for every non-collinear triple $p,q,r$ from $S$,
 or there is a non-collinear triple $p_1,q_1,r_1$ such that $|A_{p_1q_1r_1}|\ge (1-(1-c)^2)|S| = (2c-c^2)|S|$.

\bigskip

\noindent {\bf Case 1:} $|B_{pqr}|\ge (1-c)^2|S|$ for all non-collinear $p,q,r$.\\
If $B_{pqr}$ is collinear,
then $V(S\backslash \Delta_{pqr})$ lies on a $4$-flat in $\R^5$ together with $V(p),V(q),V(r)$,
so $(S\backslash \Delta_{pqr})\cup \{p,q,r\}$ is contained in a conic; denote this conic by $C_{pqr}$.
Then $S$ is contained in the union of $C_{pqr}$ and $\Delta_{pqr}$, with $p,q,r\in C_{pqr}$.
Note that if $C_{pqr}$ is reducible, then one of its lines must be part of $\Delta_{pqr}$; in this case we let $C_{pqr}$ denote the other line, so that we can say that $C_{pqr}$ is an irreducible curve of degree one or two.
We show that there are relatively few triples for which $S\subset C_{pqr}\cup \Delta_{pqr}$ can hold.

Let $p_0,q_0,r_0$ be one such triple (see Figure \ref{fig:fig3}).
 Since the set $S$ is by assumption not contained in a conic, there is at least one point of $S$ not in $C_{p_0q_0r_0}$.
 We fix one such point $u\in S\backslash C_{p_0q_0r_0}$ and assume without loss of generality that $u\in \overline{p_0r_0}$.
Note that $|S\cap C_{p_0q_0r_0}| = |B_{pqr}| \ge 5$.
Suppose $p_0,q_0,r$ is another triple for which $B_{p_0q_0r}$ is collinear and $S\subset C_{p_0q_0r}\cup \Delta_{p_0q_0r}$.
We must have $S\cap C_{p_0q_0r_0}\subset C_{p_0q_0r}$, so $|C_{p_0q_0r}\cap C_{p_0q_0r_0}|\ge |S\cap C_{p_0q_0r_0}| \ge 5$.
By B\'ezout's inequality and the fact that the curves are irreducible (of degree one or two),
this implies that $C_{p_0q_0r}= C_{p_0q_0r_0}$, which gives $r\in C_{p_0q_0r_0}$.
Moreover, we must have $u\in \Delta_{p_0q_0r}$,
so $u$ must lie on the line $\overline{p_0r}$ or the line $\overline{q_0r}$.
Conversely, this means $r$ must lie on $\overline{up_0}$ or $\overline{uq_0}$.
But each of these lines intersects $C_{p_0q_0r_0}$ in at most one point apart from $p_0,q_0$,
so there are only two possibilities for $r$.
Therefore, for any choice of $p_0,q_0$, there are at most three choices of $r$ so that we get a triple as above.
This means that there are at most $2|S|^2$ triples $p,q,r$ for which $B_{pqr}$ is collinear.

\begin{figure}[ht]
\centering
\includegraphics[height=2.5in]{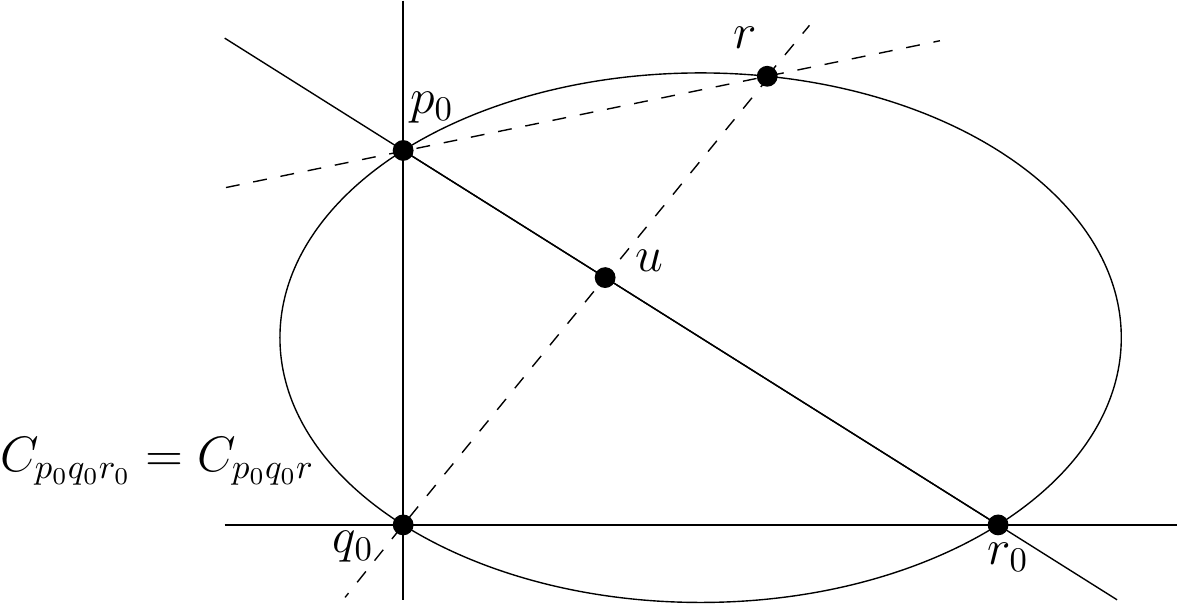}
 \caption{The argument in Case 1.}\label{fig:fig3}
\end{figure}

By Lemma \ref{lem:noncollinear},
there are $\frac{1}{6}(1-c)(|S|^3-|S|^2)$ non-collinear triples $p,q,r$ from $S$,
and for at most $2|S|^2$ of these $B_{pqr}$ is collinear.
For all other non-collinear triples we obtain at least $\frac{1}{2}(1-c)^2|S|-M$ ordinary conics containing $p,q,r$ (for this we also need $|B_{pqr}|\geq L$, but we can assume that $|S|$ is sufficiently large so that $(1-c)^2|S| \ge L$).
Any conic occurs $\binom{5}{3}$ times, once for every triple out of its five points.
Altogether, there are at least
\[\frac{1}{\binom{5}{3}}\left(\frac{1}{2}(1-c)^2|S|\right)\cdot \left(\frac{1}{6}(1-c)|S|^3\right) - O_c(|S|^3)= \frac{1}{120}(1-c)^3|S|^4 - O_c(|S|^3)\]
ordinary conics for $S$ in Case 1.

\bigskip

\noindent {\bf Case 2:}
$|S\cap \Delta_{p_1q_1r_1}|\ge (2c-c^2)|S|$ for some $p_1,q_1,r_1$.\\
We show that in this case we can use the specific structure of $\Delta_{p_1q_1r_1}$ to choose many non-collinear $p,q,r$ such that $B_{pqr}$ is reasonably large and not collinear (see Figure \ref{fig:fig4}).
If we relabel the points so that $|\overline{p_1q_1}\cap S|\geq |\overline{p_1r_1}\cap S|\geq |\overline{q_1r_1}\cap S|$,
then $\overline{p_1q_1}$ and $\overline{p_1r_1}$ must each contain at least $\frac{1}{2}(c-c^2)|S|$ points of $S$.
Indeed, $\overline{p_1q_1}$ contains at most $c|S|$ points, which implies that $\overline{p_1r_1}$ and $\overline{q_1r_1}$ together contain at least $(2c-c^2)|S| - c|S| = (c-c^2)|S|$ points of $S$, at least half of which must lie on $\overline{p_1r_1}$.
Let $L_{p_1q_1}=(\overline{p_1q_1}\cap S)\backslash \{p_1,q_1\}$ and $L_{p_1r_1}=(\overline{p_1r_1}\cap S)\backslash\{p_1,r_1\}$, so we have $|L_{p_1q_1}|,|L_{p_1r_1}|\ge\frac{1}{2}(c-c^2) |S|-2$.

\begin{figure}[ht]
\centering
\includegraphics[height=2.5in]{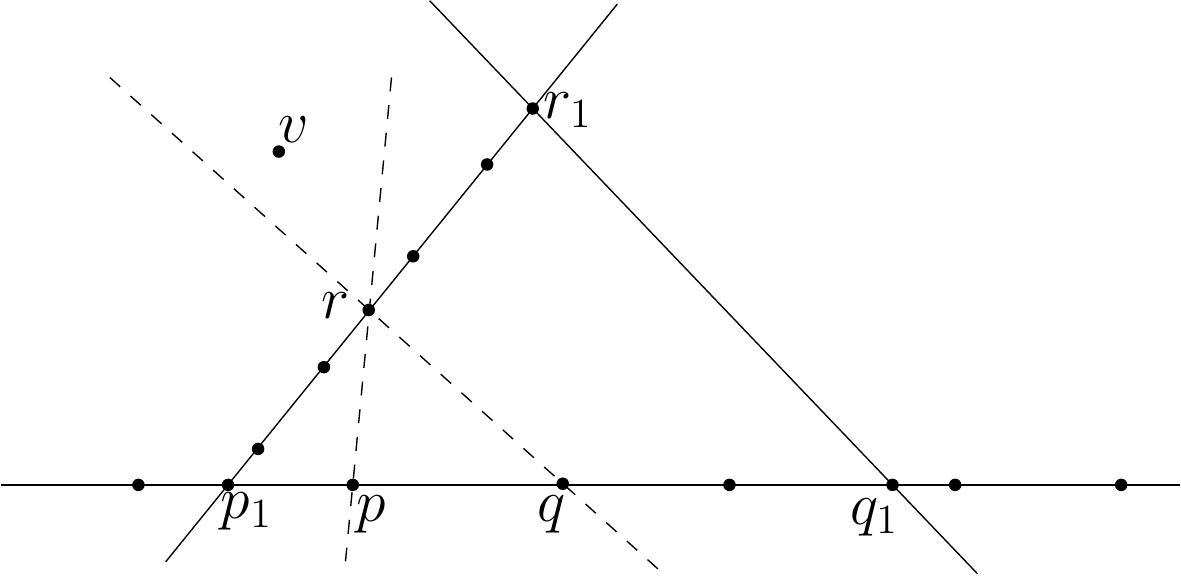}
 \caption{The argument in Case 2.}\label{fig:fig4}
\end{figure}

Since $S$ is not contained in a conic,
we can fix  $v\in S$ outside $\overline{p_1q_1}\cup \overline{p_1r_1}$ (see Figure \ref{fig:fig4}).
For any choice of $p,q\in L_{p_1q_1}$ and $r\in L_{p_1r_1}$,
the triple $p,q,r$ is not collinear,
and
\begin{align}\label{eq:Bpqr}
|B_{pqr}|\geq |L_{p_1r_1}| \geq \frac{1}{2}(c-c^2)|S|-2.
\end{align}
We will choose $p,q\in L_{p_1q_1}$ and $r\in L_{p_1r_1}$ so that $v\notin \Delta_{pqr}$.
We claim that for such a triple, $B_{pqr}$ is not collinear.
Otherwise, $v$ and $L_{p_1r_1}$ together with $p,q,r$ would lie on a conic $C$ (as in Case 1).
This is impossible,
because $|L_{p_1r_1}|\ge 3$ implies that $C$ contains
$\overline{p_1r_1}$,
and that the other points on $C$ are  collinear.
But $p,q,v$ are not collinear, a contradiction.

We now count the triples $p,q,r$ with  $p,q\in L_{p_1q_1}$, $r\in L_{p_1r_1}$, and $v\notin \Delta_{pqr}$.
For every $r\in L_{p_1r_1}$, there is at most one $p\in L_{p_1q_1}$ such  that $v\in \overline{pr}$, so there are at most $|L_{p_1q_1}|$ pairs $p,q\in L_{p_1q_1}$ such that $v\in\Delta_{pqr}$.
Hence there are at most $|L_{p_1q_1}|\cdot|L_{p_1r_1}|$ triples $p,q,r$ with $p,q\in L_{p_1q_1}$,
and $r\in L_{p_1r_1}$ such that $u\in\Delta_{pqr}$.
Therefore, the number of remaining triples is at least
\[|L_{p_1r_1}|\cdot \binom{|L_{p_1q_1}|}{2}-|L_{p_1q_1}|\cdot|L_{p_1r_1}| = \frac{1}{2}|L_{p_1q_1}|\cdot |L_{p_1r_1}|\cdot \left(|L_{p_1q_1}| -5\right) \]
\[\ge \frac{1}{2}\left(\frac{1}{2}(c-c^2) |S|-2\right)^3 - O_c(|S|^2)
= \frac{1}{16}c^3(1-c)^3|S|^3 -O_c(|S|^2).\]
For every such triple we obtain $\frac{1}{2}|B_{pqr}|-M \geq \frac{1}{4}(c-c^2)|S|-M-1$ ordinary conics (using \eqref{eq:Bpqr}, and assuming $|S|$ is sufficiently large).
Accounting for the multiplicity $\binom{5}{3}$, we obtain the lower bound
\[\frac{1}{640}c^4(1-c)^4|S|^4-O_c(|S|^3)\]
on the number of ordinary conics determined by $S$, which completes the proof.
\end{proof}

\section{A construction with few ordinary conics}\label{sec:construction}

An \emph{elliptic curve} is a cubic that is irreducible and nonsingular.
On an elliptic curve $E$ we have an operation $\oplus$ and an identity element $\mathcal{O}$ which make the set of points of $E$ into a group.
See \cite{GT} for an introduction to this group law.
The key property is that three points $p,q,r\in E$ are collinear if and only if $p\oplus q\oplus r= \mathcal{O}$.
Note that $p\oplus p\oplus q=\mathcal{O}$ should be interpreted as $p$ and $q$ lying on a line that is tangent to $E$ at $p$.

The following fact\footnote{The third author thanks Mehdi Makhul and Josef Schicho for making him aware of this fact.} can be found in \cite[Theorem 9.2]{W}.
We give our own proof, because below we will need a version for reducible cubics.


\begin{lemma}\label{lem:coconiconelliptic}
Let $E$ be an elliptic curve in $\R^2$.
Then six distinct points $p,q,r,s,t,u\in E$ are co-conic if and only if $p\oplus q\oplus r\oplus s\oplus t\oplus u= \mathcal{O}$.
\end{lemma}
\begin{proof}
We use Chasles's theorem (also known as the Cayley-Bacharach theorem; see \cite{GT}), which states that given two cubics that intersect in nine distinct points, any cubic that passes through eight of these points also passes through the ninth.

Let $p,q,r,s,t,u\in E$ be contained in a conic $C$.
The three lines $\overline{pq}, \overline{rs}, \overline{tu}$ form a cubic that intersects $E$ in the nine points $p,q,r,s,t,u, \ominus p\ominus q,\ominus r\ominus s, \ominus t\ominus u$.
These points need not be distinct,
but we can assume that they are by perturbing them and afterwards using continuity.
Hence Chasles's theorem applies to these nine points.
The conic $C$ together with the line through $\ominus p\ominus q$ and $\ominus r\ominus s$ forms a cubic, which by Chasles must also pass through $\ominus t\ominus u$.
Since $C$ already has six points of $E$, and cannot contain more by B\'ezout and the fact that $E$ is irreducible,
$\ominus t\ominus u$ must lie on the line through $\ominus p\ominus q$ and $\ominus r\ominus s$.
So $\ominus p\ominus q,\ominus r\ominus s, \ominus t\ominus u$ are collinear, and we have
\begin{align}\label{eq:sumofthree}
p\oplus q\oplus r\oplus s\oplus t\oplus u = \ominus\left((\ominus p\ominus q) \oplus (\ominus r\ominus s)\oplus (\ominus t\ominus u)\right) = \mathcal{O}.
\end{align}

Conversely, suppose that $p\oplus q\oplus r\oplus s\oplus t\oplus u= \mathcal{O}$, which implies that $\ominus p\ominus q,\ominus r\ominus s, \ominus t\ominus u$ are on a line $L$.
Let $C$ be the conic determined by $p,q,r,s,t$.
Then $C\cup L$ is a cubic that passes through eight of the points $p,q,r,s,t,u, \ominus p\ominus q,\ominus r\ominus s, \ominus t\ominus u$,
so by Chasles we have $u\in C\cup L$.
Since $L$ already contains three points of $E$, we must have $u\in C$, so $p,q,r,s,t,u$ are co-conic.
\end{proof}

\begin{corollary}\label{cor:coconiconelliptic}
Let $E$ be an elliptic curve in $\R^2$.
Five distinct points $p,q,r,s,t\in E$ satisfy $p\oplus p\oplus q\oplus r\oplus s\oplus t = \mathcal{O}$ if and only if they lie on a conic $C$, with $C$ tangent to $E$ at $p$.
\end{corollary}
\begin{proof}
This follows from Lemma \ref{lem:coconiconelliptic} by continuity.
We can approximate the five points with six co-conic points on $E$, with two of them approaching $p$.
Then the conic containing the six points approaches a conic tangent to $E$ at $p$.
The fact that the group operation is continuous implies the statement.
\end{proof}

Just like the fact that $p,q,r$ are collinear if and only if $p\oplus q\oplus r= \mathcal{O}$ leads to constructions with few ordinary lines (see \cite{GT}),
Lemma \ref{lem:coconiconelliptic} lets us construct point sets with few ordinary conics.
This is the first part of Theorem \ref{thm:construction}.

\begin{theorem}\label{thm:construction1}
There exist arbitrarily large finite sets $G\subset \R^2$, not contained in a conic and with no four on a line, that determine at most $\frac{1}{24}|G|^4+O(|G|^3)$ ordinary conics.
\end{theorem}
\begin{proof}
We can choose an elliptic curve $E$ and a finite subgroup $G\subset E$ of any size (see \cite{GT}).
If $C$ is an ordinary conic for $G$, containing the five points $p,q,r,s,t$, then the point $-(p\oplus q\oplus r\oplus s\oplus t)$ must also be in $G$, by the definition of a subgroup, and it must be on $C$ by Lemma \ref{lem:coconiconelliptic}.
Since $C$ is ordinary, this point $-(p\oplus q\oplus r\oplus s\oplus t)$ must equal one of the five points, and $C$ must be tangent to $E$ at that point.
Conversely, given five points $p,q,r,s,t\in G$ such that $-(p\oplus q\oplus r\oplus s\oplus t)\in \{p,q,r,s,t\}$,
they must determine an ordinary conic for $G$ that is tangent to $E$ at $-(p\oplus q\oplus r\oplus s\oplus t)$.
Therefore, the ordinary conics for $G$ correspond exactly to those five-tuples $p,q,r,s,t$ such that $-(p\oplus q\oplus r\oplus s\oplus t)\in \{p,q,r,s,t\}$.

Choose one point $p$ from $G$, and three other points $q,r,s$.
Then $t = -(p\oplus p\oplus q\oplus r\oplus s)$ is the unique point such that $p\oplus p\oplus q\oplus r\oplus s\oplus t= \mathcal{O}$.
On the other hand, for any ordinary conic for $G$,
$p$ is uniquely determined as the point where $C$ is tangent to $E$, while there are $\binom{4}{3}$ ways to choose $q,r,s$.
Thus the number of ordinary conics determined by $G$ equals
$\frac{1}{4}|G|\cdot \binom{|G|}{3} =\frac{1}{24}|G|^4 + O(|G|^3)$.
\end{proof}

We now move towards the second part of Theorem \ref{thm:construction},
which uses a similar construction on a reducible cubic.
As explained in \cite{GT}, one can define a ``quasi-group law'' on reducible cubics; see \cite[Section 7]{GT} for details.
It does not quite make the set of points of the cubic into a group, but it comes close enough to allow constructions like in Theorem \ref{thm:construction1}.
We consider only the case of a reducible cubic that is the union of an irreducible conic $C$ and a disjoint line $L$.
By \cite[Proposition 7.3]{GT}, there are bijective maps $\psi_C:\mathbb{R}\slash\mathbb{Z}\to C, \psi_L:\mathbb{R}\slash\mathbb{Z}\to L$ such that $\psi_C(a),\psi_C(b), \psi_L(c)$ are collinear if and only if $a+b+c=0$.
We have the following equivalent of Lemma \ref{lem:coconiconelliptic}.

\begin{lemma}\label{lem:coconiconreducible}
Let $C$ be an irreducible conic and $L$ a disjoint line in $\R^2$.
Let $p,q\in L$ and $r,s,t,u\in C$ be distinct  points.
Then $p,q,r,s,t,u\in C\cup L$ are co-conic if and only if
\[\psi_L^{-1}(p)+\psi_L^{-1}(q) + \psi_C^{-1}(r)+ \psi_C^{-1}(s)+ \psi_C^{-1}(t)+ \psi_C^{-1}(u) =0.\]
The same holds when two of the six points coincide, with the conic containing all five points being tangent to $C$ or $L$ at the repeated point.
\end{lemma}
\begin{proof}
The proof is essentially the same as that of Lemma \ref{lem:coconiconelliptic}, except that the group operations take place in $G$.
We also need to choose the three lines carefully, and to deal with the reducibility of the cubic.

Let $p,q,r,s,t,u\in C\cup L$ be contained in a conic $C'$.
The three lines $\overline{pr},\overline{qs}, \overline{tu}$, form a cubic that intersects $C\cup L$ in the nine points $p,q,r,s,t,u, \ominus p\ominus q,\ominus r\ominus s, \ominus t\ominus u$.
The conic $C'$ together with the line through $\ominus p\ominus q$ and $\ominus r\ominus s$ forms a cubic, which by Chasles must also pass through $\ominus t\ominus u$.
The conic $C'$ has six points of $C\cup L$, and by B\'ezout it cannot contain more,
unless $C'=C$ or $L\subset C'$.
But $C'=C$ is impossible because $C'$ contains $p,q\not \in C$,
and $L \subset C'$ is impossible because then the other line in $C'$ would contain four points $r,s,t,u$ from the irreducible conic $C$.
We conclude that $\ominus t\ominus u$ lies on the line through $\ominus p\ominus q$ and $\ominus r\ominus s$,
and we can finish as in the proof of Lemma \ref{cor:coconiconelliptic}.
The converse is then straightforward, and the last statement of the lemma follows in the same way as Corollary \ref{cor:coconiconelliptic}.
\end{proof}

We now get a construction much like in Theorem \ref{thm:construction1} (with more awkward counting).

\begin{theorem}
There exist arbitrarily large finite sets $H\subset \R^2$, not contained in a conic and with $|H|/2$ points on a line, that determine at most $\frac{7}{384}|H|^4+O(|H|^3)$ ordinary conics.
\end{theorem}
\begin{proof}
Set $H_1=\psi_C(G)$, $H_2 = \psi_L(G)$, and $H = H_1\cup H_2$. We have $|H| = 2|G|$.
The ordinary conics for $H$ correspond to the following two types of five-tuples.

\begin{itemize}

\item One point $p\in L$ and four points $q,r,s,t\in C$ such that
\[\psi_L^{-1}(p)+\psi_L^{-1}(p) + \psi_C^{-1}(q)+ \psi_C^{-1}(r)+ \psi_C^{-1}(s)+ \psi_C^{-1}(t) =0.\]

\item Two points $p,q\in L$ and three points $r,s,t\in C$  such that
\[\psi_L^{-1}(p)+\psi_L^{-1}(q) + \psi_C^{-1}(r)+ \psi_C^{-1}(r)+ \psi_C^{-1}(s)+ \psi_C^{-1}(t) =0.\]

\end{itemize}
We count the number of five-tuples of both types.
For the first type, we can choose $p\in L$ in $|G|$ ways,
and $q,r,s\in C$ in $\binom{|G|}{3}$ ways; $t$ is then determined.
One such five-tuple occurs with multiplicity $\binom{4}{3}$, so the number of five-tuples of the first type is
\[\frac{1}{\binom{4}{3}}|G|\cdot \binom{|G|}{3} = \frac{1}{24} |G|^4 +O(|G|^3)= \frac{1}{384}|H|^4 + O(|H|^3). \]

For the second type,
we can choose $p,q\in L$ in $\binom{|G|}{2}$ ways,
$r\in C$ in $|G|$ ways, and then $s\in C$ in at most $|G|-1$ ways; $t$ is then determined.
One such five-tuple occurs with multiplicity two (note that $r$ is determined as the point on $C$ where the ordinary conic is tangent; only $s$ and $t$ could be interchanged).
Hence the number of five-tuples of the second type is
\[\frac{1}{2}\binom{|G|}{2}\cdot |G| \cdot (|G|-1) = \frac{1}{4}|G|^4 +O(|G|^3) = \frac{1}{64}|H|^4+O(|H|^4).\]
Altogether we get the number of ordinary conics stated in the theorem.
\end{proof}

We note that Green and Tao \cite{GT} define quasi-group laws on any cubic, including on a union of three lines.
But the two types used above (elliptic curves and conics with lines) are the only ones for which the underlying group has arbitrarily large finite subgroups\footnote{The \emph{acnodal cubic} also has this property, but the resulting construction would be similar to that for elliptic curves.};
the other cases do not have this property over $\R$.
Because of this, these other curves would not give good constructions (as is the case for ordinary lines).

\end{document}